\documentclass[11pt]{amsart}

\usepackage{amsfonts,amsmath,latexsym,amssymb,verbatim,amsbsy,amsthm,mathabx}
\usepackage{dsfont,bm}

\usepackage[top=1in, bottom=1in, left=1in, right=1in]{geometry}

\usepackage[dvipsnames]{xcolor}

\usepackage[colorlinks=true, pdfstartview=FitV, linkcolor=RoyalBlue,citecolor=ForestGreen, urlcolor=blue]{hyperref}

\theoremstyle{plain}
\newtheorem{THEOREM}{Theorem}[section]

\newtheorem{theorem}[THEOREM]{Theorem}

\newtheorem{lemma}[THEOREM]{Lemma}
\newtheorem{proposition}[THEOREM]{Proposition}

\theoremstyle{definition}

\newtheorem{definition}[THEOREM]{Definition}

\theoremstyle{remark}

\newtheorem{remark}[THEOREM]{Remark}
\newtheorem{claim}[THEOREM]{Claim}

\newcommand{\thm}[1]{Theorem~\ref{#1}}
\newcommand{\lem}[1]{Lemma~\ref{#1}}

\newcommand{\prop}[1]{Proposition~\ref{#1}}

\def \a {\alpha}
\def \b {\beta}
\def \g {\gamma}
\def \d {\delta}
\def \e {\varepsilon}
\def \f {\varphi}
\def \k {\kappa}
\def \l {\lambda}
\def \n {\nabla}
\def \s {\sigma}
\def \t {\tau}
\def \th {\theta}

\def \Th {\Theta}
\def \O {\Omega}

\def \barPhi { \widebar{\Phi}}
\def \barx {\bar{x}}
\def \barv {\bar{v}}
\def \barth {\bar{\th}}
\def \barTh {\widebar{\Th}}
\def \barX {\widebar{X}}
\def \barV {\widebar{V}}

\def \cA {\mathcal{A}}

\def \cC {\mathcal{C}}
\def \cD {\mathcal{D}}

\def \cF {\mathcal{F}}
\def \cG {\mathcal{G}}
\def \cH {\mathcal{H}}

\def \cK {\mathcal{K}}

\def \cO {\mathcal{O}}
\def \cP {\mathcal{P}}
\def \cQ {\mathcal{Q}}
\def \cR {\mathcal{R}}
\def \cS {\mathcal{S}}

\def \cW {\mathcal{W}}

\newcommand{\N}{\ensuremath{\mathbb{N}}}   
\newcommand{\R}{\ensuremath{\mathbb{R}}}   

\newcommand{\E}{\ensuremath{\mathbb{E}}}
\renewcommand{\S}{\ensuremath{\mathbb{S}}} 

\renewcommand{\geq}{\geqslant}
\renewcommand{\ge}{\geqslant}
\renewcommand{\leq}{\leqslant}

\DeclareMathOperator{\supp}{supp} %
\DeclareMathOperator{\diam}{diam} %
\DeclareMathOperator{\Id}{Id} %
\def \lan {\langle}
\def \ran {\rangle}

\def \p {\partial}

\def \ss {\subset}

\def \HI {H\"older inequality}

\renewcommand{\geq}{\geqslant}
\renewcommand{\ge}{\geqslant}
\renewcommand{\leq}{\leqslant}

\def \dx  {\, \mbox{d}x}

\def \dX  {\, \mbox{d}X}
\def \dV  {\, \mbox{d}V}

\def \dy  {\, \mbox{d}y}
\def \dz  {\, \mbox{d}z}
\def \dr  {\, \mbox{d}r}
\def \ds  {\, \mbox{d}s}
\def \dw  {\, \mbox{d}w}
\def \dth  {\, \mbox{d}\th}
\def \dTh  {\, \mbox{d}\Th}
\def \de  {\, \mbox{d}\eta}

\def \dv  {\, \mbox{d} v}
\def \dbarV  {\, \mbox{d} \barV}
\def \dbarX  {\, \mbox{d} \barX}
\def \dbarTh  {\, \mbox{d} \barTh}

\def \ddt  {\frac{\mbox{d\,\,}}{\mbox{d}t}}

\def \Lip {\mathrm{Lip}}

\def \tv {\tilde{v}}
\def \tw {\tilde{w}}
\def \tu {\tilde{u}}
\def \ip {{\it\Pi}}
\def\vp {v^{\ip}}
\def\wp {w^{\ip}}
\def\up {u^{\ip}}
\begin{document}

\title{Propagation of Chaos for the Cucker-Smale Systems under heavy tail communication}

\author{Vinh Nguyen}

\author{Roman Shvydkoy}

\address{Department of Mathematics, Statistics and Computer Science, University of Illinois at Chicago, 60607}

\email{vnguye66@uic.edu }

\email{shvydkoy@uic.edu}

\subjclass{92D25, 35Q35}

\date{\today}

\keywords{Cucker-Smale, collective behavior, propagation of chaos, mean-field limit}

\thanks{\textbf{Acknowledgment.}  
	The work of RS was supported in part by NSF
	grants DMS-1813351 and DMS-2107956.}

\begin{abstract}
In this work we study propagation of chaos for solutions of the Liouville equation for the classical discrete Cucker-Smale system. Assuming that the communication kernel satisfies the heavy tail condition -- known to be necessary to induce exponential alignment -- we obtain a linear in time convergence rate of the $k$-th marginals  $f^{(k)}$ to the product of $k$ solutions of the corresponding Vlasov-Alignment equation, $f^{\otimes k}$. Specifically, the following estimate holds in terms of Wasserstein-2 metric
\begin{equation}\label{e:abst}
 \cW_2(f^{(k)}_t,f^{\otimes k}_t) \leq C k^{1/2}  \min\left\{ 1, \frac{t}{\sqrt{N}} \right\}.
\end{equation}
For systems with the Rayleigh-type friction and self-propulsion force, we obtain a similar result for sectorial solutions. Such solutions are known to align exponentially fast via the method of Grassmannian reduction, \cite{LRS-friction}. We recast the method in the kinetic settings and show that the bound \eqref{e:abst} persists but with the quadratic dependence on time. 

In both the forceless and forced cases, the result represents an improvement over the exponential bounds established earlier in the work of Natalini and Paul, \cite{NP2021-orig}, although those bounds hold for general kernels. The main message of our work is that flocking dynamics improves the rate considerably.

\end{abstract}

\maketitle

\section{Background and main results}

 One of the fundamental questions of the mathematical theory of large systems of particles is a derivation and formal justification of the corresponding kinetic models. Among the many systems describing collective phenomena this question has been successfully settled for the Cucker-Smale model describing the basic  mechanism of alignment \cite{CS2007a,CS2007b}:
\begin{equation}\label{e:CS}
\left\{\begin{split}
\dot{x}_i & = v_i, \qquad x_i(0) = x_i^0 \in \R^n \\
\dot{v}_i & =  \frac{1}{N} \sum_{j=1}^N \phi(x_i - x_j)(v_j - v_i), \quad  v_i(0) = v_i^0 \in \R^n.
\end{split}\right.
\end{equation}
Here $\phi$ is a non-negative non-increasing smooth communication kernel.  The corresponding Vlasov-Alignment equation is given by
\begin{equation}\label{e:VE}
\p_t f + v\cdot \n_x f + \n_v \cdot (f F(f)) = 0, \quad f(0) = f_0: \R^{2n} \to \R_+,
\end{equation}
where 
\[
F(f)(x,v) = \int_{\R^{2n}} \phi(x-y) (w-v) f(y,w,t) \dy \dw.
\]
A formal derivation of \eqref{e:VE} via the BBJKY hierarchy was performed   in Ha and Tadmor \cite{HT2008}, and rigorously via the mean-field limit in Ha and Liu \cite{HL2009}.  

The hierarchy approach is based upon the classical idea of  propagation of chaos, which postulates that the particles $(x_1,v_1,\ldots,x_N,v_N)$ whose joint probability distribution $f^N$ is given by the solution to the Liouville transport equation 
\begin{equation}\label{e:LE}
\p_t f^{N} + \sum_{i=1}^N v_i \cdot \n_{x_i} f^N + \sum_{i=1}^N \n_{v_i} \cdot(f^N F_i^N) = 0,
\end{equation}
would gradually decorrelate as $N\to \infty$ if initially so
\begin{equation}\label{e:LEic}
f^N(0) = f_0^{\otimes N}, \qquad f_0 : \R^{2n} \to \R_+,
\end{equation}
and their individual distributions would evolve according to \eqref{e:VE}.
In other words, 
\begin{equation}\label{e:prop}
\lan f^{N}, \f_1\otimes \ldots \otimes \f_k \otimes 1 \otimes \dots \otimes 1 \ran \to \prod_{j=1}^k \lan f, \f_j \ran, \qquad \f\in C_b(\R^{2nk}).
\end{equation}

The mean-field limit on the other hand, is based upon weak convergence of a sequence of empirical measures built from solutions to \eqref{e:CS},
\[
\mu^N = \frac{1}{N} \sum_{j=1}^N \d_{x_i(t)} \otimes \d_{v_i(t)} \to f.
\]
In fact, a more detailed analysis done in \cite{Ha-stability,Sbook} establishes Lipschitz continuity of measure-valued solutions to  \eqref{e:VE} with respect to the Wasserstein metric,
\[
\cW_p(\mu'_t, \mu''_t) \leq C(t) \cW_p(\mu'_0,\mu''_0).
\]

It is well-known, however, that propagation of chaos and the mean-field limit (in a somewhat more specific sense) are equivalent, see Sznitman \cite{Sznitman}. In fact, \eqref{e:prop} holds if and only if for any $\f \in \Lip(\R^{2n})$ one has
\begin{equation}\label{e:Elim}
E_\f(t) = \int_{\R^{2nN}} \left| \frac1N \sum_{j=1}^N \f(x_i(t), v_i(t)) - \lan f_t, \f \ran \right|^2 f_0^{\otimes N} \dX_0 \dV_0 \to 0,
\end{equation}
where $X_0,V_0$ are the initial conditions for the characteristic flow $\{ x_i(t), v_i(t) \}_{i=1}^N$. Note that initially $E_\f(0) \to 0$ by a direct verification.  Technically, since not every initial ensemble $X_0,V_0$ in the support of $f_0^{\otimes N}$ forms an empirical measure weakly close to $f_0$, the limit \eqref{e:Elim} does not directly follow from \cite{HL2009,Ha-stability,Sbook}. However, one can restore it using similar estimates on the deformation of the flow-map of \eqref{e:CS} and coupling with the characteristics of \eqref{e:VE}.

In any case, Snitzman's general principle seems to provide little quantitative information on the rate of propagation in \eqref{e:prop} as it avoids using any specificity of the system at hand. 
For stochastically forced systems, the work of Bolley, Ca\~nizo and Carrillo  \cite{BCC2011} establishes such a quantitative estimate on the Wasserstein-2 distance:
\begin{equation}\label{e:W2stoch}
\cW_2(f^{(k)}_t,f^{\otimes k}_t) \leq C \sqrt{\frac{k}{N^{e^{-Ct}}}}.
\end{equation}
Recently, Natalini and Paul addressed the deterministic case in \cite{NP2021-orig} and with additional chemotaxis forces in \cite{NPchemo}. For the forceless system, the estimate 
carries exponential dependence in time,
\begin{equation}\label{e:W2det}
\cW_2(f^{(k)}_t,f^{\otimes k}_t) \leq C e^{\d t} \sqrt{\frac{k}{N}}.
\end{equation}

The estimates \eqref{e:W2stoch}, \eqref{e:W2det} are finite-time bounds in spirit, in the sense that they do not take into account any flocking long-time behavior of the system. In this present work we raise the question: can one improve upon the time dependence in the deterministic case \eqref{e:W2det} when the system is known to flock exponentially fast?  It is the result that goes back to Cucker and Smale \cite{CS2007a} and improved and extended in  \cite{HL2009, HT2008, CFRT2010} that the system \eqref{e:CS} with a heavy tail radial communication,
\begin{equation}\label{e:ht}
\int_0^\infty \phi(r) \dr = \infty
\end{equation}
aligns with an exponential rate. Let us give a quantitative summary of this result for future reference, see also \cite{Sbook} for details.

\begin{proposition}\label{p:flock} Suppose $\phi$ satisfies \eqref{e:ht}. For any solution to \eqref{e:CS} with initial data in $(X_0,V_0)$ in a compact domain $\O \ss \R^{2nN}$ the following flocking estimates hold:
\begin{equation}\label{e:flockCS}
\sup_{t>0} \max_{i,j=1,\ldots,N} |x_i - x_j| ={D} < \infty, \quad \max_{i,j=1,\ldots,N} |v_i - v_j| \leq A_0 e^{-t \phi({D})},
\end{equation}
where $A_0$ is the initial velocity fluctuation and ${D}$ depends only on the initial diameter of the flock and $\phi$. 

Similarly, for any solution $f$ to \eqref{e:VE} with initial compact support one has
 \begin{equation}\label{e:flockkinCS}
\sup_{t>0} \diam \supp f_t ={D} < \infty, \quad \max_{(x',v'),(x'',v'')\in \supp f_t} |v' - v''| \leq A_0 e^{-t \phi({D})}.
\end{equation}
\end{proposition}

With the use of this additional flocking information we will improve the estimate \eqref{e:W2det} to being linear in time.

\begin{theorem}\label{t:main} Suppose $\phi$ satisfies \eqref{e:ht}, and let 
 $f_0\in C^1_0(\R^{2n})$ be an initial distribution with a compact support. Let $f^N$ be the solution to \eqref{e:LE}-\eqref{e:LEic}, while $f$ be the solution to \eqref{e:VE}. Then there exists a constant $C$ which depends only on $\diam(\supp f_0)$ and $\phi$ such that for all $N\in \N$, $k\leq N$, and $t\geq 0$ one has
\begin{equation}\label{e:chaos}
 \cW_2(f^{(k)}_t,f^{\otimes k}_t) \leq C \sqrt{k}  \min\left\{ 1, \frac{t}{\sqrt{N}} \right\}.
\end{equation}
\end{theorem}

Our general methodology relies on the same classical coupling method, which compares characteristic flow of the original system \eqref{e:CS} to $N$ copies of the flow-map of the kinetic transport  \eqref{e:VE}, but it  differs from  \cite{NP2021-orig} in two aspects. First, we run the entire argument from the Lagrangian point of view, which gives a direct access to characteristics and the flocking estimates. This is closer in spirit to the original mean-field approach of \cite{HL2009} or  \cite{BCC2011} in stochastic settings. Second, we rely on the flocking information of  \prop{p:flock} to extract a crucial stabilizing exponential factor in the estimation of kinetic energy, see \eqref{e:Kfinal}. The linear time dependence here comes primarily from growth of the potential energy, and it seems not to be removable within the given framework.

Next, we consider the same problem in the context of systems forced with self-propulsion and Rayleigh-type friction force with variable characteristic parameters $\th$:
\begin{equation}\label{e:CSR}
\begin{cases}
	\dot{x}_i= v_i,\\
	\dot{v}_i = \dfrac{1}{N} \sum\limits_{j=1}^{N}\phi(x_i-x_j)(v_j - v_i)+ \s v_i(\th_i-|v_i|^p), \\
	\dot{\th}_i = \dfrac{\k}{N} \sum\limits_{j=1}^{N}\phi(x_i-x_j)(\th_j - \th_i),
\end{cases} \qquad(x_i,v_i,\th_i) \in \R^n \times \R^n \times \R_+,
\end{equation}
where $\k >0$ is a coupling coefficient and $p>0$. This model is relevant in the study of systems of agents with a tendency to adhere to their preferred characteristic speeds $\th_i$, see \cite{Ha-friction,LRS-friction}.  The recent study \cite{LRS-friction} introduced a general method of Grassmannian reduction that allows to prove flocking for solutions with velocities confined to a sector $\Sigma$ of opening $<\pi$, so-called sectorial solutions, see \prop{p:flockR} below. We give an extension of this method to the corresponding kinetic Vlasov equation in \prop{p:flockRkin} and use it to prove propagation of chaos for the forced system \eqref{e:CSR}. Specifically, we show that is the kernel $\phi$ satisfies a power low bound with heavy tail, see \eqref{e:kernelpower} below,  then for any sectorial solutions satisfy
\begin{equation}\label{e:chaosRF}
 \cW_2(f^{(k)}_t,f^{\otimes k}_t) \leq C \sqrt{k}  \min\left\{ 1, \frac{t^2}{\sqrt{N}} \right\}.
\end{equation}
To achieve this bound we employ monotonicity of the force to control the ``bad" self-propulsion component.  The ultimate effect of its presence, however, is reflected in the quadratic dependence on time in \eqref{e:chaosRF}.

 In the case $\k =0$ our analysis gives no additional improvement over \eqref{e:W2det}. However, the derived kinetic equation can interpreted as a model of opinion dynamics of a large population which takes into account fixed conviction values $\th$. See Remark~\ref{r:k=0} for more discussion.

\section{Propagation of chaos for the forceless system}\label{sec:classicalCS}

In this section we focus on establishing propagation of chaos for the pure Cucker-Smale system \eqref{e:CS}.  So, to fix the notation let us consider a solution $f^N$ to the full Liouville equation \eqref{e:LE} with the product initial condition \eqref{e:LEic} on the configuration space $(X,V) \in \R^{2nN}$. We can assume without loss of generality that $f_0$ is a probability distribution. The forces $F_i^N$'s are given by the Cucker-Smale system
\[
F_i^N(X,V) = \frac1N \sum_{j=1}^N \phi(x_i - x_j)(v_j - v_i).
\]
Due to symmetries of the forces, the solution will remain symmetric with respect to permutations of pairs $(x_i,v_i)$ for all time.

We define the $k$-th marginal as usual by
\begin{equation}\label{def:1st_mgl}
f^{(k)}_t(x_1,v_1,\ldots,x_k,v_k) = \int_{\R^{2n(N-k)}} f^N_t(x_1,v_1,\ldots,x_N,v_N) \dx_{k+1}\ldots \dv_N.
\end{equation}

Let us introduce various characteristic maps that will be used in the proof. We denote by 
\[
\Phi^N_t = (x_1(t),v_1(t),\ldots,x_N(t),v_N(t)): \R^{2n N} \to \R^{2n N}
\]
the flow-map of the Liouville equation \eqref{e:LE}, in other words these are solutions to the agent-based system \eqref{e:CS}.  Then, $f^N_t$ at any time $t>0$ is a push-forward of the initial distribution by $\Phi^N_t$,
\begin{equation}\label{e:push}
f^N_t = \Phi_t^N \sharp f_0^{\otimes N}.
\end{equation}

Now, denote by 
\[
\barPhi_t = (\barx(t),\barv(t)): \R^{2n} \to \R^{2n}
\]
the flow-map of the Vlasov equation \eqref{e:VE}, i.e.
\begin{equation}\label{e:kinCS}
\left\{\begin{split}
\dot{\barx} & = \barv, \qquad \barx(0) = \barx_0 \\
\dot{\barv} & =  \int_{\R^{2n}} \phi(\barx-y) (w-\barv) f(y,w,t) \dy \dw, \quad \barv(0) = \barv_0,
\end{split}\right.
\end{equation}
and by 
\[
\barPhi^{\otimes N}_t = (\barx_1(t),\barv_1(t),\ldots,\barx_N(t), \barv_N(t)): \R^{2nN} \to \R^{2nN}
\]
the direct product of $N$ copies of $\barPhi_t$'s.  Thus,
\begin{equation}\label{e:pushbar}
f_t =  \barPhi_t \sharp f_0, \qquad  f^{\otimes N}_t =  \barPhi_t^{\otimes N} \sharp f_0^{\otimes N}.
\end{equation}

The proof of \thm{t:main} can be reduced to establishing the following estimate
\begin{equation}\label{e:wassaux}
 \int_{R^{2nN}}  | \Phi_t^N(X_0,V_0) - \barPhi_t^{\otimes N}(X_0,V_0)|^2 \,f_0^{\otimes N}(X_0,V_0) \dX_0 \dV_0 \leq  C \min\{ N, t^2\}.
\end{equation}
Indeed, let us recall that the Wasserstein-2 distance between two probability measures $\mu',\mu''$  on $\R^{2nk}$ can be defined in probabilistic sense as 
\[
\cW^2_2(\mu,\bar{\mu}) = \inf \E [ |Z - \widebar{Z} |^2 ],
\]
where the infimum is taken over $\R^{2nk}$-valued random variables $Z,\widebar{Z}$ defined on any probability space with distributions given by $\mu$ and $\bar{\mu}$, respectively. To measure the distance between $f^{(k)}_t$ and $f^{\otimes k}_t$ we can pick the probability space $\R^{2nN}$ with measure $f_0^{\otimes N}(X_0,V_0) \dX_0 \dV_0$, and random variables given by any selection of $k$ coordinates of $\Phi_t^N$ and $\barPhi_t^{\otimes N}$, respectively, because their probability distributions relative to the chosen base space are exactly $f^{(k)}_t$ and $f^{\otimes k}_t$ according to \eqref{e:push} and \eqref{e:pushbar}. 

So, let us denote by $\Sigma_N^k$ is the set of all ordered subsets of $[1,\ldots,N]$ of size $k$. Clearly, its cardinality is ${N\choose{k}}$.  Then, for any $\s \in \Sigma_N^k$,
\[
 \cW^2_2(f^{(k)}_t,f^{\otimes k}_t)  \leq  \int_{R^{2nN}} \sum_{i=1}^k | (x_{\s(i)},v_{\s(i)})-(\barx_{\s(i)},\barv_{\s(i)})|^2 \,f_0^{\otimes N}(X_0,V_0) \dX_0 \dV_0
\]
Summing up over all $\s \in \Sigma_N^k$, we obtain
\[
{N\choose{k}} \cW^2_2(f^{(k)}_t,f^{\otimes k}_t)  \leq   \int_{R^{2nN}} \sum_{\s \in \Sigma_N^k}  \sum_{i=1}^k | (x_{\s(i)},v_{\s(i)})-(\barx_{\s(i)},\barv_{\s(i)})|^2 \,f_0^{\otimes N}(X_0,V_0) \dX_0 \dV_0.
\]
Observe that in the double sum inside the integral each coordinate will be repeated ${N-1\choose{k-1}}$ times. So, 
\[
{N\choose{k}} \cW^2_2(f^{(k)}_t,f^{\otimes k}_t)  \leq  {N-1\choose{k-1}} \int_{R^{2nN}}  \sum_{i=1}^N | (x_{i},v_{i})-(\barx_{i},\barv_{i})|^2 \,f_0^{\otimes N}(X_0,V_0) \dX_0 \dV_0.
\]
Simplifying and using \eqref{e:wassaux}, we obtain
\[
\cW^2_2(f^{(k)}_t,f^{\otimes k}_t)  \leq C k  \min\left\{ 1, \frac{t^2}{N} \right\},
\]
as desired.

To establish  \eqref{e:wassaux} let us break the expression under the integral into potential and kinetic part,
\begin{equation}\label{}
\cP = \frac12 \int_{R^{2nN}}  | X_t - \barX_t |^2 \, f_0^{\otimes N} \dX_0 \dV_0, \qquad \cK = \frac12 \int_{R^{2nN}}  | V_t - \barV_t  |^2\, f_0^{\otimes N}\dX_0 \dV_0.
\end{equation}
Here, $X_t,V_t$ and $\barX_t, \barV_t$ denote the corresponding components of $\Phi_t^N$ and $\barPhi_t^{\otimes N}$, respectively.
Clearly,
\begin{equation}\label{e:PK}
\ddt \cP \leq 2 \cP^{1/2} \cK^{1/2}.
\end{equation}

Let us now write out the equation for the kinetic part,
\begin{equation*}\label{e:kin_part}
\begin{split}
\ddt \cK & =  \int_{R^{2nN}}\sum_{i=1}^N (v_i - \barv_i) \cdot \left( \frac1N \sum_{j=1}^N \phi(x_i - x_j)(v_j - v_i) - \int_{R^{2n}}\phi(\barx_i-y) (w-\barv_i) f(y,w,t) \dy \dw \right) \\
& \hspace{3in} \times f_0^{\otimes N} \dX_0 \dV_0 \\
& = A+B+C,
\end{split}
\end{equation*}
where 
\[
\begin{split}
A& = \int_{R^{2nN}}\sum_{i=1}^N (v_i - \barv_i) \cdot \frac1N \sum_{j=1}^N [\phi(x_i - x_j)- \phi(\barx_i - \barx_j)](v_k - v_i) \, f_0^{\otimes N} \dX_0 \dV_0\\
B & = \int_{R^{2nN}} \sum_{i=1}^N (v_i - \barv_i) \cdot \frac1N \sum_{j=1}^N \phi(\barx_i - \barx_j) [(v_j-\barv_j) - (v_i-\barv_i)] \, f_0^{\otimes N} \dX_0 \dV_0\\
C& = \int_{R^{2nN}}\sum_{i=1}^N (v_i - \barv_i) \cdot \left( \frac1N \sum_{j=1}^N \phi(\barx_i - \barx_j)(\barv_j - \barv_i) - \int_{R^{2n}}\phi(\barx_i-y) (w-\barv_i) f(y,w,t) \dy \dw \right) \\
&\hspace{3in} \times f_0^{\otimes N}\dX_0 \dV_0.
\end{split}
\]

Let us start with $C$. Apply the \HI\ first
\[
\begin{split}
C^2 \leq 2\cK \int_{R^{2nN}} \sum_{i=1}^N &\left|\frac1N \sum_{j=1}^N \phi(\barx_i - \barx_j)(\barv_j - \barv_i) - \int_{R^{2n}}\phi(\barx_i-y) (w-\barv_i) f(y,w,t) \dy \dw  \right|^2  \\
&\hspace{3in}   \times f_0^{\otimes N} \dX_0 \dV_0
\end{split}
\]
switching back to the Eulerian coordinates, whereby $\barx_i,\barv_i$ become dummy variables, we continue
\[
= 2\cK  \int_{R^{2nN}} \sum_{i=1}^N \left|\frac1N \sum_{j=1}^N \phi(\barx_i - \barx_j)(\barv_j - \barv_i) - \int_{R^{2n}}\phi(\barx_i-y) (w-\barv_i) f(y,w,t) \dy \dw  \right|^2 f_t^{\otimes N} \dbarX \dbarV
\]
All these terms, due to symmetry are independent of $i$. According to \cite[Lemma 3.3]{NP2021-orig}, and our flocking estimate \eqref{e:flockkinCS}, each can be estimated by
\[
\frac4N \sup_{(\barx',\barv'),(\barx'',\barv'')\in \supp f_t} |\phi(\barx'-\barx'')(\barv'-\barv'')|^2 \leq \frac{c}{N}e^{-\d t},
\]
Thus,
\[
C \leq ce^{-\d t} \cK^{1/2}.
\]

Turning back to $A$, we use the smoothness of the kernel and exponential flocking estimates \eqref{e:flockCS}, 
\[
\begin{split}
|A| & \leq c e^{-\d t} \sqrt{\cK} \left( \int_{R^{2nN}}\sum_{i=1}^N \left[ \frac1N \sum_{j=1}^N (|x_i -\barx_i|+ |x_j- \barx_j|)\right]^2 \, f_0^{\otimes N} \dX_0 \dV_0\right)^{1/2} \\
& \leq c e^{-\d t} \sqrt{\cK} \left( \int_{R^{2nN}}\sum_{i=1}^N \left[ |x_i -\barx_i|^2 + \frac1N \sum_{j=1}^N |x_j- \barx_j |^2 \right] \, f_0^{\otimes N} \dX_0 \dV_0\right)^{1/2}\\
& \leq  c e^{-\d t} \sqrt{\cK} \left(2 \int_{R^{2n N}} \left[\sum_{i=1}^N |x_i -\barx_i|^2 \right] \, f_0^{\otimes N} \dX_0 \dV_0\right)^{1/2}\\
& = c e^{-\d t} \sqrt{\cK} \sqrt{\cP}. 
\end{split}
\]
Finally, one can see that $B$ contributes a negative term,
\[
\sum_{i=1}^N (v_i - \barv_i) \cdot \frac1N \sum_{j=1}^N \phi(\barx_i - \barx_j) [(v_j-\barv_j) - (v_i-\barv_i)]  = \frac1N \sum_{i,j=1}^N\phi(\barx_i - \barx_j)( (v_i - \barv_i)\cdot (v_j -\barv_j) - |v_i-\barv_i|^2 )
\]
and symmetrizing,
\begin{multline*}
 = \frac12 \frac1N \sum_{i,j=1}^N\phi(\barx_i - \barx_j) (- |v_j-\barv_j|^2+ 2(v_i - \barv_i)\cdot (v_j-\barv_j) - |v_i-\barv_i|^2 ) \\
  = - \frac12 \frac1N\sum_{i,j=1}^N\phi(\barx_i - \barx_j)| (v_j- \barv_j) - (v_i- \barv_i)|^2 \leq 0.
 \end{multline*}

Collecting all of the above we obtain
\begin{equation}\label{e:Kfinal}
\ddt \cK \leq ce^{-\d t} ( \cK^{1/2} + \cK^{1/2} \cP^{1/2}).
\end{equation}

Denoting $x = 1+ \cP^{1/2}$, $y = \cK^{1/2}$ we obtain the system
\begin{equation}\label{e:xysys}
\dot{x} \leq y,\ x_0 = 1; \qquad \dot{y} \leq ce^{-\d t} x,\ y_0 = 0.
\end{equation}
\begin{claim} Any non-negative solution to \eqref{e:xysys} obeys an estimate $x\leq 1+ Ct$, $y \leq C \min\{1,t\}$, where $C = C(c,\d)$.
\end{claim}
To see that let us fix an $\e>0$ to be determined later and compute
\[
\ddt ( \e x^2 + y^2) \leq 2 xy(\e + ce^{-\d t}) \leq \sqrt{\e}(\e x^2 +y^2) + \frac{c}{\sqrt{\e}} e^{-\d t}( \e x^2 + y^2).
\]
Thus,
\[
\e x^2 + y^2 \leq \e \exp\left\{ \sqrt{\e}t + \frac{1}{\sqrt{\e}\d} \right\}.
\]
Setting $\e = \d^2$, we can see that the growth rate of $x$ does not exceed $\d/2$, $x \lesssim e^{\d t/2}$. Plugging this into $y$-equation we obtain $\dot{y} \lesssim e^{-\d t/2}$. This proves the bound on $y$, and then solving for $x$, $x \leq 1 + Ct$.

Going back to the energies, we obtain
\[
\cK \leq C\min\{1,t^2\}, \qquad \cP \leq Ct^2.
\]
Due to the global bound on the support of the flock \eqref{e:flockCS}, \eqref{e:flockkinCS}, we also have $\cP \leq C N$. Thus,
\[
\cP \leq C \min\{ N, t^2\}.
\]
Consequently, we obtain the required
\[
\cK + \cP \leq C\min\{ N, t^2\}.
\]

\section{Propagation of chaos for forced system}

Using the basic energy estimates obtained in the previous section, we will now extend the result to the system with friction forces \eqref{e:CSR} and $\k>0$. It is well-known that the flocking behavior of solutions to \eqref{e:CSR}, even with constant $\th_i = 1$ does not always hold even for global kernels $\phi \geq c_0 >0$. The example exhibited in \cite{Ha-friction} shows misalignment dynamics when the initial configuration is symmetric $x_1 = -x_2$ and velocities are aimed in the opposite directions $v_1 = -v_2$.  The work \cite{LRS-friction} proved that this is, in a sense, the only situation when no flocking occurs. As long as the initial condition is {\em sectorial}, meaning that all $v_i(0) \in \Sigma$, where $\Sigma$ is an open conical sector of opening less than $\pi$, then the solutions align exponentially fast. 
\begin{proposition}[\cite{LRS-friction}]\label{p:flockR}
Suppose that 
\begin{equation}\label{e:kernelpower}
\phi(r) \geq \dfrac{\lambda}{(1+r^2)^{\b/2}},\quad  \l >0, \quad \b \leq 1.
\end{equation}
For any sectorial solution to \eqref{e:CSR} there exists $v_\infty\in \R^n$ and $\th_\infty >0$ with $|v_\infty|^p = \th_\infty$, such that one has
\begin{align}
    &\max_{i = 1,\ldots, N}(|v_i - v_\infty |+|\th_i - \th_\infty |) \leq Ce^{-\d t},\\
    & \sup_{t>0}\max\limits_{i,j = 1,\ldots, N} |x_i - x_j|= D < \infty.
\end{align}
\end{proposition}

It is within the context of sectorial solutions that we will cast the propagation of chaos result. But first we establish a similar flocking estimates for solutions of the corresponding kinetic model.

\subsection{Grassmannian reduction for Vlasov-alignment equation}

Let us denote $\O = \R^n \times \R^n \times \R_+$.  The Vlasov equation corresponding to \eqref{e:CSR} is given by
 \begin{equation}\label{e:VER}
\p_t f + v\cdot \n_x f + \n_v \cdot (f F(f)) + \n_v\cdot(f R) + \n_\th\cdot(f\Th(f)) = 0, \quad (x,v,\th) \in \O, \quad t>0,
\end{equation}
subject to the initial condition 
\begin{equation}\label{e:VEic_Rayleigh}
f(x,v,\th,0) = f_0(x,v,\th),
\end{equation}
where
\[
\begin{split}
F(f)(x,v,\th) &= \int_{\O} \phi(x-y) (w-v) f(y,w,\eta, t) \dz \dw\de,\\ 
R(x,v,\th) & = \s v(\th-|v|^p), \quad \s > 0, \quad p>0,\\
\Th(f)(x,v,\th) &= \k \int_{\O} \phi(x-y) (\eta-\th) f(y,w,\eta,t) \dz \dw\de.
\end{split}
\]

In this section, we will prove a similar flocking result for the sectorial solutions of \eqref{e:VER}. Let us define what they in the kinetic context.
\begin{definition}
A solution $f$ to \eqref{e:VER} is called {\em sectorial} if there exists a conical region $\Sigma$ lying on one side of a hyperplane, i.e. with conical opening less than $\pi$ such that $v\in \Sigma$ for any $v$ in the velocity support of $f$, $(x,v,\th)\in \supp f$ for some $x,\th$.
\end{definition}

Since the equation \eqref{e:VER}  is rotation invariant, it will be convenient to assume that our solution belong the upper-half space:  there exists $\e>0$ such that 
\begin{equation}\label{def:sectorial}
    v_n  \geq \e |v| , \quad \forall (x,v,\th)\in \supp f, 
\end{equation}
By the weak maximum principle discussed below in Remark~\ref{r:wmp}, it follows that if $f$ is sectorial initially, then it will remain so for all time and the velocity support will lie in the same sector $\Sigma$.
 
Let us state our main result now.

\begin{proposition}\label{p:flockRkin}
Suppose the kernel satisfies  \eqref{e:kernelpower}.
For any sectorial solution $f$ to \eqref{e:VER} with initial compact support one has
\begin{align}\label{kinCSR_flock_est}
\sup\limits_{t>0}\diam \supp f_t  <\infty, 
\end{align}
and there exist $v_\infty \in \R^n$, $\th_\infty \in \R_+$, with $|v_\infty|^p = \th_\infty$ such that
\begin{align}
\max\limits_{(x,v,\th) \in \supp f_t}(|\th - \th_\infty| + |v-v_\infty|) \leq c e^{- \d t }. \label{kin_B_est}
\end{align}
\end{proposition}

As in the discrete case the proof is based on examination of kinetic characteristics of the equation given by 
\begin{equation}\label{e:kinCSR}
\left\{\begin{split}
\dot{x} & = v, \qquad x(0) = x_0 \\
\dot{v} & =  \int_{\O} \phi(x-y) (w-v) f(y,w,\eta,t) \dy \dw \de +\s v(\th-|v|^p), \quad v(0) = v_0\\
\dot{\th} &= \k \int_{\O} \phi(x-y) (\eta-\th) f(y,w,\eta,t) \dy \dw \de, \quad \th(0) = \th_0.
\end{split}\right.
\end{equation}
Let us denote
\begin{align*}
    \cD(t) &= \max_{(x,v,\th),(x',v',\th') \in \supp f}|x-x'|,\\
    \cA(t) &= \max_{(x,v,\th),(x',v',\th')\in \supp f}|v-v'|,\\
    \cQ(t) &= \max_{(x,v,\th),(x',v',\th')\in \supp f}|\th-\th'|,
\end{align*}
\begin{align*}
M &= \int_\Omega f(x,v,\th,t)\dx \dv \dth,\quad \th_\infty = \dfrac{1}{M}\int_\O \th f(x,v,\th,t)\dx \dv  \dth,\\
 \th_+(t) &= \max\limits_{(x,v,\th)\in \supp f_t} \th,\qquad \th_-(t) = \min\limits_{(x,v,\th)\in \supp f_t} \th.
\end{align*}
Then we have
\begin{equation}
    \ddt \cD  \leq \cA,
    \end{equation}
and
\begin{equation}\label{theta}
    \ddt \cQ  \leq - \phi (\cD) \cQ.
\end{equation}
It is not hard to see that $\th_+$ is decreasing and $\th_-$ is increasing. Thus, 
\begin{equation}
   \th_+(t) \leq \th^*, \quad \th_-(t) \geq \th_* \quad \forall t\ge 0,  
\end{equation}
where
$\th^* = \th_+(0) $ and $  \th_* = \th_-(0)$.

Before proceed let us discuss boundedness of velocity support of $f$ and the weak maximum principle.

\begin{lemma}[boundedness]\label{v:bdd}
There exists a constant $C$ which depends on the initial data such that for any $(x,v,\th) \in \supp f_t$, one has
\begin{equation}
    |v(t)| \leq C, \qquad \forall t > 0.
\end{equation}
\end{lemma}
\begin{proof}
Let
\[
|v_+|(t) = \max\limits_{(x,v,\th)\in \supp f_t} |v|.
\]
At time $t$, let $\ell \in (\R^d)^*, |\ell| =1$, $(x,\, v,\th) \in \supp f_t $ such that  $|v_+| = \ell(v)$. Then, by Rademacher's Lemma,
\[
\begin{split}
\ddt |v_+| &= \int_{\O}\phi(x - z)\,\ell (w - v)f(z,w,\eta,t)\,\dz\dw\de  +\s\ell(v)(\th-|v|^p)\\ 
  &\leq \s |v_+|(\th^*-|v_+|^p).
\end{split}
\]
Hence, if $\th^* \leq |v_+|^p$ then 
\[
|v_+|(t) \leq |v_+|(0) \quad \forall t> 0.
\] 
Otherwise, we have
\begin{equation*}
   \ddt |v_+|^p \leq \s p |v_+|^p  (\th^*-|v_+|^p).
\end{equation*}
Solving the above ODI gives 
\begin{equation}\label{v:upper_bound}
    |v_+|(t) \leq \dfrac{\sqrt[p]{\th^*} e^{\s\th^* t}}{(c+e^{\s p\th^* t})^{1/p}} = \sqrt[p]{\th^*} +\cO (e^{-\s\th^* t}),
\end{equation}
where $c$ is a positive constant depending on initial data. Thus, $|v_+|(t)$ is bounded for all $t > 0$. 
\end{proof}
 \begin{lemma}[weak maximum principle]\label{l:wmp} If for a given functional $\ell \in (\R^n)^*$, all velocity vectors $v_0$ that lie in the support of the initial flock, $(x_0,v_0,\th_0)\in \supp f_0$, satisfy
 \begin{equation*}
     \ell(v_0) \geq 0, 
 \end{equation*}
 then at any positive time 
 \begin{equation*}
    \ell(v) \geq 0, \qquad \forall t>0, \ (x,v,\th) \in \supp f_t.
 \end{equation*}
 \end{lemma}
\begin{proof}
 At time $t$, let 
\[\ell(v) = \min\limits_{ (z,w,\eta)\in \supp f} \ell(w).
\]
By Rademacher's Lemma,
\begin{equation*}
  \ddt \ell(v) = \int_{\O}\phi(x - z)\,\ell (w - v)f(z,w,\eta,t)\,\dz\dw\de  +\s\ell(v)(\th-|v|^p)
  \geq \s\ell(v)(\th_*-|v|^p).
\end{equation*}
 Then by Lemma \ref{v:bdd} we get
\begin{equation*}
    \ddt \ell(v)\geq c\,\ell(v), 
\end{equation*}
where $c$ is constant. Solving this ODI we obtain the desired conclusion,
\begin{equation*}
    \ell(v) \geq \ell(v_0) e^{ct} \geq 0, \quad \forall t>0.
\end{equation*}
\end{proof}

\begin{remark}\label{r:wmp}
By the weak maximum principle we note that if the support of $f_0$ in $v$ lies in the convex sector defined by
\[
\Sigma_{\cF} = \bigcap\limits_{\ell \in \cF}\left \{v \in \R^n: \ell(v) \geq 0\right \},
\]
where $\cF$ is an arbitrary set of linear functionals on $\R^n$, then the velocity support of $f_t$ will be confined to that sector for all time.
Since the system \eqref{e:kinCSR} is invariant under rotations, without loss of generality we can assume that the support of $f_0$ in $v$ lies above the hyperplane $\ip_n = \{v_n = 0\}$, where $v_n$ is the $n$-th coordinate of vector $v$.
\end{remark}

\begin{lemma}\label{V:bdd_below}
For any sectorial solution $f$ to \eqref{e:VER} there exists a positive constant $c_0$ depending on the initial data such that
\begin{equation}\label{v:lower_bound}
    |v| \geq c_0, \quad \forall (x,v,\th) \in \supp f_t.
\end{equation}
\end{lemma}
\begin{proof}
At time $t$, let $(x,v,\th)$ be a minimizer for $\min \limits_{(x,v,\th)\in \supp f_t} v_n$. Then 
\begin{equation}\label{e:vnODI}
    \ddt v_n =  \int_{\O}\phi(x - z) (w_n - v_n)f(z,w,\eta,t)\,\dz\dw\de   +  \s v_n(\th-|v|^p) 
  \geq \s v_n(\th_*-\e^{-p}v_n^p).
\end{equation}
If $\th_* \leq \e^{-p}v_n^p$ then 
\[
|v| \geq \e\sqrt[p]{\th_*}.
\]
Otherwise, solving \eqref{e:vnODI} we get
\[
v_n \geq \dfrac{\e\sqrt[p]{\th_*} e^{\s\th_* t}}{(c+e^{p\s \th_*   t})^{1/p}},
\]
where $c$ is a positive constant which depends on the initial data. Then the lemma follows.
\end{proof}
\begin{remark}
Lemma \ref{V:bdd_below} tells us that for a sectorial solution $f$, $\supp f(x,\cdot,\th)$ stays away from the origin. Then, by Lemma \ref{v:bdd}, it implies that $\supp f(x,\cdot,\th)$ is contained in a sector. Lemma \ref{V:bdd_below} also implies that for any sectorial solution $f$ one has
\begin{equation}\label{vm_bdd}
    |v_{-}|(t) \geq c_0, \quad \forall t > 0,
\end{equation}
where $|v_-| (t) = \min \limits_{(x,v,\th) \in \supp f} |v(t)|$.
\end{remark}
\medskip

\begin{proof}[Proof of \prop{p:flockRkin}]
From now on we consider a sectorial solution $f$ to the system \eqref{e:VER}. Denoting $\tilde{r}= \dfrac{r}{|r|}$ for any vector $r \in \R^n$. One has
\begin{align}
\ddt \tv =  \dfrac{1}{|v|}\left(\Id - \dfrac{v}{|v|}\otimes\dfrac{v}{|v|}\right)\dot{v} =  \int_{\O}\dfrac{|w|}{|v|}\phi (x-z)(\Id - \tv\otimes\tv)\tw\, f(z,w,\eta,t)\dz\dw\de.
\end{align}
Here, we used $(\Id - \tv\otimes\tv)v = 0$. 

\medskip
Denoting by $\widehat{(v,u)}$ the angle between two vectors $v$ and $u$, then
$\cos \widehat{(v,u)} = \tv\cdot\tu$. Thus, if $(x,v,\th), (y,u,\zeta)$ are the solutions to \eqref{e:kinCSR} with respect to the initial conditions $(x_0,v_0,\th_0), (y_0,u_0,\zeta_0)$, respectively, then
\begin{equation}\label{e:cos}
\begin{split}
    \ddt \cos \widehat{(v,u)} &= \int_{\O}\dfrac{|w|}{|v|}\phi (x-z)[\cos\widehat{(u,w)} - \cos\widehat{(v,u)}\cos\widehat{(v,w)}]f (z,w,\eta,t)\dz\dw\de \\
    &\quad + \int_{\O}\dfrac{|w|}{|u|}\phi (y-z)[\cos\widehat{(v,w)} - \cos\widehat{(v,u)}\cos\widehat{(u,w)}]f (z,w,\eta,t)\dz\dw\de.
\end{split}
\end{equation}
Note that if $v,u,$ and $w$ are three vectors lying in the same two dimenstional plane and 
\begin{equation}\label{sec_limit}
    \widehat{(v,u)} = \widehat{(v,w)} + \widehat{(w,u)} < \pi - \d \quad\text{ for some } \d > 0,
\end{equation}
then the followings hold:
\begin{align*}
   \cos\widehat{(u,w)} - \cos\widehat{(v,u)}\cos\widehat{(v,w)} &= \cos \left(\widehat{(v,u)} - \widehat{(v,w)}\right) - \cos\widehat{(v,u)}\cos\widehat{(v,w)}\nonumber\\
   & = \sin\widehat{(v,u)}\sin\widehat{(v,w)} \geq 0,\\
   \cos\widehat{(v,w)} - \cos\widehat{(v,u)}\cos\widehat{(u,w)} &\geq 0,
\end{align*}
\begin{align*}
 \cos\widehat{(u,w)} + \cos\widehat{(v,w)} = \cos\dfrac{\widehat{(v,u)}}{2}\cos\dfrac{\widehat{(u,w)} - \widehat{(v,w)}}{2}\geq \left(\cos\dfrac{\pi - \d}{2}\right)^2.  
\end{align*}
Therefore, if the support of $f$ in $v$ is on a two dimensional plane and \eqref{sec_limit} is satisfied, then by Lemma \ref{v:bdd} , Lemma \ref{V:bdd_below} and \eqref{e:cos}, one has
\begin{align*}
 \ddt \cos \widehat{(v,u)} &\geq c\phi(\cD)\int_{\O}\left(\cos\widehat{(u,w)} + \cos\widehat{(v,w)}\right)\left( 1 - \cos\widehat{(v,u)}\right)f(z,w,\eta,t)\dz\dw\de\nonumber\\
 &\geq c\phi(\cD)\left( 1 - \cos\widehat{(v,u)}\right).
\end{align*}
Equivalently,
\begin{equation}\label{e:angle}
    \ddt \left(1 -\cos \widehat{(v,u)}\right) \leq - c\phi(\cD)\left( 1 - \cos\widehat{(v,u)}\right).
\end{equation}
Now let $\ip$ be a fixed two dimensional plane which contains the $v_n$-axis. Denoting by $v^{\ip}$ the projection of any $v \in \supp f$ onto $\ip$. Projecting the second equation in \eqref{e:kinCSR} onto $\ip$ we have the following equation:
\begin{equation}\label{e:v_proj}
\dot{v}^{\ip} = \int_{\O} \phi(x-z) (\wp-\vp) f(z,w,\eta,t) \dz \dw\de +\s \vp(\th-|v|^p)
\end{equation}
Therefore, we can write the equation for $\cos \widehat{(\vp,\up)}$ as follows:
\begin{equation}\label{e:cos_proj}
  \begin{split}
    \ddt \cos \widehat{(\vp,\up)} &= \int_{\O}\dfrac{|\wp|}{|\vp|}\phi (x-z)[\cos\widehat{(\up,\wp)} - \cos\widehat{(\vp,\up)}\cos\widehat{(\vp,\wp)}]f (z,w,\eta,t)\dz\dw\de \\
    &\quad + \int_{\O}\dfrac{|\wp|}{|\up|}\phi (y-z)[\cos\widehat{(\vp,\wp)} - \cos\widehat{(\vp,\up)}\cos\widehat{(\up,\wp)}]f (z,w,\eta,t)\dz\dw\de.
\end{split}  
\end{equation}
Let us denote $\cG (1,n-1)$ the space of all two dimensional subspaces of $\R^n$ which contain $v_n$-axis. Since $\cG (1,n - 1)$ can be identified with $1$-Grassmannian manifold of $\R^{n-1}$ which is compact, we can define
\begin{equation}
    \g^{2D} = \max\limits_{\substack{ \ip \in\, \cG (1,n-1)\\
    (x,v,\th),(y,u,\zeta) \,\in \,\supp f
    }} \widehat{(\vp,\up)}.
\end{equation}
We note that 
\[
\g^{2D} \leq \pi - \d \quad \text{ for some } \d > 0.
\]
Since the $n$-th coordinate of any $v\in \supp f$ does not change when it is projected onto $\ip$, $|\vp|$ is still bounded above and below by positive constants. Therefore, choosing a maximizing triple $\ip, u, v$ for $\widehat{(\vp,\up)}$, from \eqref{e:cos_proj} we deduce that 
\begin{equation}\label{e:2Dg}
    \ddt(1- \cos\g^{2D}) \leq - c\phi(\cD)((1- \cos\g^{2D}).
\end{equation}
Denoting 
\begin{equation*}
\g = \max\limits_{(x,v,\th),(y,u,\zeta) \,\in\, \supp f} \widehat{(u,v)}. 
\end{equation*}
\begin{claim}\label{g}
We have $\g \leq \g^{2D}$.
\end{claim}
\begin{proof}[Proof of Claim \ref{g}]
For any $(x,v,\th),(y,u,\zeta) \in \supp f$, consider the two dimensional subspace $\ip = \text{span} \{e_n, \tu - \tv\}$ where $e_n = (0, \ldots, 0, 1)$. We have $\ip \in \cG(1,n-1)$ and $\tu - \tv = \tu^{\ip} - \tv^{\ip}$. By the law of cosines, we get
\begin{equation*}
\begin{split}
2(1-\cos\widehat{(u,v)}) = |\tu - \tv|^2 = |\tu^{\ip} - \tv^{\ip}|^2& = 2|\tu^{\ip} |^2(1-\cos\widehat{(u^{\ip},v^{\ip})})\\
&\leq 2 (1-\cos\widehat{(u^{\ip},v^{\ip})}).
\end{split}
\end{equation*}
It implies that for any $(x,v,\th),(y,u,\zeta) \in \supp f$ there exists $\ip \in \cG(1,n-1)$ such that $\widehat{(u,v)} \leq \widehat{(u^{\ip},v^{\ip})}$. Therefore, the claim is followed.
\end{proof}

\begin{remark}\label{cosg-g2D}
Claim \ref{g} and the inequality \eqref{e:2Dg} imply that if $\cD(t) \leq D < \infty$ then
\[
\begin{split}
   1 - \cos \g \leq 1 - \cos \g^{2D} \lesssim e^{-c\phi({D}) t}.
\end{split}
\]
\end{remark}
Now setting
\begin{equation*}
    \cR = \max\limits_{(x,v,\th), (y,u,\zeta)\in \supp f}\dfrac{|v|^2}{|u|^2}.
\end{equation*}
Suppose that $(x,v,\th), (y,u,\zeta)$ maximize $\cR$  at time $t$, we have
\begin{align}\label{R}
        \ddt \cR &= \dfrac{2}{|u|^2} \left[ \int_{\O} \phi(x-z) (v\cdot w-|v|^2) f(z,w,\eta,t) \dz \dw \de +\s |v|^2(\th-|v|^p)\right]\nonumber\\
        &\quad - \dfrac{2|v|^2}{|u|^4}\left[ \int_{\O} \phi(y-z) (u\cdot w-|u|^2) f(z,w,\eta,t) \dz \dw \de +\s |u|^2(\zeta-|u|^p)\right]\nonumber\\
        & = \dfrac{2}{|u|^2} \int_{\O} \phi(x-z) (v\cdot w-|v|^2) f(z,w,\eta,t) \dz \dw\de\\
        & \quad + \dfrac{2|v|^2}{|u|^4} \int_{\R^{2d}} \phi(y-z) (|u|^2 -u\cdot w) f(z,w,\eta,t) \dz \dw\de + 2\s\cR (\th - \zeta + |u|^p - |v|^p)\nonumber
    \end{align}
Since $u,v$ maximize $\cR$,  we have $v\cdot w - |v|^2 \leq |v|(|w| -|v|) \leq 0$ for all $w \in \supp f$. Hence, the first term on the right hand side of \eqref{R} is nonpositve. For the second term, we have
\[
|u|^2 -u\cdot w = |u|^2 - |u||w|\cos \widehat{(u,w)} \lesssim 1 -\cos\g.
\]
Note that $\cR$ is bounded from above and below, hence,
\[
2\s\cR (\th -\zeta + |u|^p - |v|^p) =2\s\cR (\th -\zeta )+ \dfrac{2\s\cR}{|u|^p}(1-\cR^{p/2})\lesssim \cQ + (1 - \cR).
\]
Therefore, there exist positive constants $c_1, c_2, c_3$ such that 
\begin{equation}\label{e:R}
    \ddt (\cR -1) \leq - c_1(\cR-1) + c_2(1 - \cos \g)+c_3\cQ.
\end{equation}

Firstly, we see that the flock diameter grows at most linearly in time,
\begin{equation}\label{diam_est}
\cD(t) \lesssim t
\end{equation}
since
\begin{equation}\label{e:diam}
  \ddt \cD(t) \leq \cA(t)
\end{equation}
and $|v|$ is bounded for all $(x,v,\th) \in \supp f$.
It is not hard to see the relation
\begin{equation}\label{A2}
    \cA^2 \lesssim (\cR -1) + (1-\cos\g).
\end{equation}
Thus, to prove an exponential alignment it suffices to show that both $(\cR -1)$ and $(1-\cos \g)$ decays exponentially fast.

\medskip
We now consider two cases for $\b$:

Case I:  $\b < 1$. Our assumption on the kernel and \eqref{diam_est} imply that 
\begin{equation}\label{phi}
\phi(\cD) \gtrsim \dfrac{1}{(1+t^2)^{\b/2}}. 
\end{equation}
Plugging it into \eqref{e:2Dg} and applying the Gr\"onwall's Lemma we get
\begin{equation}\label{e:cosg}
    1 - \cos\g \leq 1 - \cos \g^{2D} \lesssim e^{-c\langle t\rangle^{1-\b}}.
\end{equation}
Plugging \eqref{phi} into \eqref{theta} and solving for $\cQ$ we also have
\begin{equation}
    \cQ \lesssim e^{-c\langle t\rangle^{1-\b}}.
\end{equation}
Combining these inequalities with \eqref{e:R} and solving for $\cR -1$ we obtain 
\begin{equation}\label{e:R-1}
    \cR -1 \lesssim e^{-c\langle t\rangle^{1-\b}}.
\end{equation}
From \eqref{e:diam}, \eqref{A2}, \eqref{e:cosg} and \eqref{e:R-1}, we have
\[
\ddt \cD \lesssim e^{-c\langle t\rangle^{(1-\b)/2}}.
\]
Solving this ODI gives 
\begin{equation}\label{finite_diam}
\cD(t) \leq D< \infty.
\end{equation}
Thus, \eqref{theta} implies that 
\[
\cQ(t) \leq \cQ(0) e^{-t\phi(D)}.
\]
Hence, $\th(t)$ aligns to $\th_\infty$ exponentially fast for all $(x,v,\th) \in \supp f$.
Due to finite flock diameter \eqref{finite_diam} and Remark \ref{cosg-g2D}, we have
\[
1 - \cos \g \lesssim e^{-c\phi({D}) t}.
\]
Putting the estimates for $\cQ$ and $(1-\cos \g)$ into \eqref{e:R} and solving for $\cR -1$ where we use the Gr\"onwall's Lemma, we obtain the exponential decay for $\cR -1$ as well. Therefore, we arrive at an alignment with an exponential rate.

Denoting by $E$ any quantity which decays exponentially fast. So far we have $|\th -\th_\infty| = E(t), |v - u| = E(t)$ for any $\th, v,u \in \supp f$. By \eqref{vm_bdd} and Lemma \ref{v:bdd}, $|v_{\pm}|(t)$ are bounded, hence, the following equations hold for $|v_{\pm}|^p(t) -\th_\infty$:
\begin{equation*}
    \ddt (|v_{\pm}|^p - \th_\infty) = (\s p|v_{\pm}|^p(\th_\infty- |v_{\pm}|^p) + E)  \quad \sim  \quad (- (|v_{\pm}|^p - \th_\infty) + E).
\end{equation*}
It follows that $|v_{\pm}|^p(t)$ converges to $\th_\infty$ exponentially fast. Therefore, from the characteristic equation for $v \in \supp f$ in \eqref{e:kinCSR} we deduce that 
\[
\ddt v = E, \quad \forall v_0 \in \supp f_0.
\]
The existence of $v_\infty$ is followed then.

Case II: $\b = 1$. In this case, we have $\phi(\cD) \gtrsim \dfrac{1}{\sqrt{1+t^2}}$, hence,
\begin{equation*}
     1 - \cos\g \leq 1 - \cos \g^{2D} \lesssim \langle t\rangle^{-\a}, \quad\text{ and } 
\end{equation*}
\begin{equation*}
   \cQ \lesssim  \langle t\rangle^{-\a}, \text{ for some } \a >0.
\end{equation*}
Therefore, 
\begin{equation*}
    \ddt (\cR -1) \lesssim - (\cR-1) + \langle t\rangle^{-\a} .
\end{equation*}
Solving this ODI we yield
\begin{equation*}
    \cR - 1 \lesssim \langle t\rangle^{-\a}.
\end{equation*}
Here we used the fact that $e^{-c t} \ast \langle t\rangle^{-\a} \sim\langle t\rangle^{-\a}$. It implies that 
\begin{equation*}
    \cA \lesssim \langle t\rangle^{-\a/2},
\end{equation*}
and hence, 
\begin{equation*}
    \cD\lesssim \langle t\rangle^{1-\a/2}.
\end{equation*}
Thus,
\begin{equation*}
    \phi (\cD) \gtrsim \phi( \langle t\rangle^{1-\a/2})\gtrsim \dfrac{1}{(1+t^2)^{\tilde{\beta}/2}} \text{ for some } \tilde{\beta} < 1.
\end{equation*}
Now we can argue exactly as in the case $\b < 1$ replacing $\b$ by $\tilde{\beta}$ to reach the conclusions of the theorem.
\end{proof}

 \subsection{Propagation of Chaos}
 Using \prop{p:flockRkin} as a key ingredient we now prove our main result for the Rayleigh-forced system.  So, let us we consider the full Liouville equation for a probability density $f^N$ on $\O^{N}$: 
\begin{equation}\label{e:LER}
\p_t f^N + \sum_{i=1}^N v_i \cdot \n_{x_i} f^N + \sum_{i=1}^N \n_{v_i} \cdot(f^N F_i^N)+ \sum_{i=1}^N \n_{v_i} \cdot(f^N R_i^N)+ \sum_{i=1}^N \n_{\th_i} \cdot(f^N \Th_i^N) = 0,
\end{equation}
subject to the initial condition 
\begin{equation}\label{e:LERic}
f^N(0) = f_0^{\otimes N}, 
\end{equation}
where $ f_0 : \O \to \R_+$ and for $(X,V,\Theta) = (x_1,\ldots,x_N, v_1,\ldots,v_N,\th_1,\ldots,\th_N)$,
\[
\begin{split}
    F_i^N(X,V,\Theta) &= \frac1N \sum_{k=1}^N \phi(x_i - x_k)(v_k - v_i) ,\\
    \Th_i^N(X,V,\Theta) &= \frac1N \sum_{k=1}^N \phi(x_i - x_k)(\th_k - \th_i),\\
    R_i^N (X,V,\Theta) &= \s v_i(\th_i-|v_i|^p).
\end{split}
\]

\begin{theorem}\label{t:mainR}
Suppose the kernel $\phi$ satisfies \eqref{e:kernelpower}. Let $f_0\in C^1_0(\O)$ be a sectorial initial distribution, and $f^N, f$ be the solutions to  to the system \eqref{e:LER} and \eqref{e:VER}, respectively. Then there exists a constant $C$ which depends only on $\diam(\supp f_0)$ and $\phi$ such that for all $N\in \N$, $k < N$, and $t\geq 0$ one has
\begin{equation}\label{e:chaos2}
 W_2(f^{(k)}_t,f^{\otimes k}_t) \leq C k^{1/2} \min\left\{ 1, \frac{t^2}{\sqrt{N}} \right\}.
\end{equation}
\end{theorem}
\begin{proof}

We introduce a similar notation for the flow-maps. Denote by 
\[
\Phi^N_t = (x_1(t),v_1(t),\th_1(t)\ldots,x_N(t),v_N(t),\th_N(t)): \O^N \to \O^{N}
\]
the flow-map of the discrete system \eqref{e:CSR} which is also the characteristic flow of \eqref{e:LER}. Then, as before,  $f^N $ is the push forward of $f_0^{\otimes N}$ under  $\Phi_t^N$, 
\[
f^N = \Phi_t^N \sharp f_0^{\otimes N}.
\]
Let also
\[
\barPhi_t = (\barx(t),\barv(t),\barth(t)): \O \to \O
\]
be the characteristic map of \eqref{e:VER}, which consists of solutions to  \eqref{e:kinCSR}. The direct product of $N$ copies will be denoted $\barPhi^{\otimes N}_t$.  Then we have 
\begin{equation}
    f = \barPhi_t \sharp f_0,\quad f^{\otimes N} =  \barPhi_t^{\otimes N} \sharp f_0^{\otimes N}.
\end{equation}

By the same logic as before the theorem reduces to establishing the bound
\begin{equation}\label{e:wassauxR}
 \int_{R^{2nN}}  | \Phi_t^N(X_0,V_0) - \barPhi_t^{\otimes N}(X_0,V_0)|^2 \,f_0^{\otimes N}(X_0,V_0) \dX_0 \dV_0 \leq  C \min\{ N, t^4\}.
\end{equation}
We split the integrand into three components:
\begin{equation}\label{}
\begin{split}
\cP &= \frac12 \int_{\O^{N}}  | X_t(X_0,V_0,\Th_0) - \barX_t(X_0,V_0,\Th_0) |^2 \, f_0^{\otimes N}(X_0,V_0,\Th_0) \dX_0 \dV_0\dTh_0,\\
\cK &= \frac12 \int_{\O^{N}}  | V_t(X_0,V_0,\Th_0) - \barV_t(X_0,V_0,\Th_0) |^2\, f_0^{\otimes N}(X_0,V_0,\Th_0) \dX_0 \dV_0\dTh_0,\\
\cC &= \frac12 \int_{\O^{N}}  | \Th_t(X_0,V_0,\Th_0) - \barTh_t(X_0,V_0,\Th_0)  |^2\, f_0^{\otimes N}(X_0,V_0,\Th_0) \dX_0 \dV_0\dTh_0.
\end{split}
\end{equation}
For the potential energy we will use the same inequality as before, \eqref{e:PK}. For $\cK$, we obtain
\[
\ddt \cK  = \cS_1 + \cS_2,
\]
where $\cS_1$ is the exact same alignment term that we handled before, but now with the use of \prop{p:flockR} and \prop{p:flockRkin},
\begin{equation}\label{s1}
    \cS_1 \leq c e^{-\d t}\cK^{1/2} (1+\cP^{1/2}).
\end{equation}
And $\cS_2$ is given by
\[
\cS_2 =  \int_{\O^{N}}\sum_{i=1}^N (v_i - \barv_i) \cdot \left(\s v_i(\th_i-|v_i|^p) - \s \barv_i(\barth_i-|\barv_i|^p)\right)f_0^{\otimes N}(X_0,V_0,\Th_0) \dX_0 \dV_0\dTh_0 .
\]

Let us write  $\cS_2$ as follows
\[
\begin{split}
\cS_2 &=  \s\int_{\O^N}\sum_{i=1}^N (v_i - \barv_i) \cdot \left( \th_i v_i  -\barth_i \barv_i\right)f_0^{\otimes N}(X_0,V_0,\Th_0) \dX_0 \dV_0\dTh_0\\
&\quad - \s\int_{\O^N}\sum_{i=1}^N (v_i - \barv_i) \cdot \left( v_i|v_i|^p) -  \barv_i|\barv_i|^p\right)f_0^{\otimes N}(X_0,V_0,\Th_0) \dX_0 \dV_0\dTh_0\\
& := J_1 - J_2.
\end{split}
\]
Since
\[
\begin{split}
 (v_i - \barv_i) \cdot \left( \th_i v_i  -\barth_i \barv_i\right)  = \dfrac{1}{2}(\th_i + \barth_i) |v_i - \barv_i|^2 + \dfrac{1}{2}(v_i- \barv_i)\cdot [(\th_i - \barth_i) (  v_i + \barv_i)],   
\end{split}
\]
one has
\[
\begin{split}
    J_1 = \dfrac{\s}{2}\int_{\O^N}\sum_{i=1}^N \left((\th_i + \barth_i) |v_i - \barv_i|^2 + (v_i-\barv_i)\cdot [(\th_i - \barth_i) (v_i + \barv_i)]\right) f_0^{\otimes N}(X_0,V_0,\Th_0) \dX_0 \dV_0\dTh_0.
\end{split}
\]
For $J_2$, since
\[
\begin{split}
      (v_i - \barv_i) \cdot \left( v_i|v_i|^p -  \barv_i|\barv_i|^p\right) 
      & = \dfrac{1}{2}( |v_i|^p+|\barv_i|^p ) |v_i - \barv_i|^2 +  \dfrac{1}{2}(|v_i|^2 - |\barv_i|^2)(|v_i|^p - |\barv_i|^p),
\end{split}
\]
and 
\[
\dfrac{1}{2}(|v_i|^2 - |\barv_i|^2)(|v_i|^p - |\barv_i|^p) \geq 0,
\]
we get
\[
-J_2 \leq - \dfrac{\s}{2}\int_{\O^N}\sum_{i=1}^N  (|\barv_i|^p + |v_i|^p) |v_i - \barv_i|^2f_0^{\otimes N}(X_0,V_0,\Th_0) \dX_0 \dV_0\dTh_0.
\]
Therefore,
\begin{align}
    \cS_2 = J_1 - J_2& \leq  \dfrac{\s}{2}\int_{\O^N}\sum_{i=1}^N  (\th_i - |v_i|^p + \barth_i -|\barv_i|^p) |v_i - \barv_i|^2f_0^{\otimes N}(X_0,V_0,\Th_0) \dX_0 \dV_0\dTh_0\nonumber\\
    & \quad + \dfrac{\s}{2}\int_{\O^N}\sum_{i=1}^N (v_i-\barv_i)\cdot (\th_i - \barth_i) (\barv_i + v_i)f_0^{\otimes N}(X_0,V_0,\Th_0) \dX_0 \dV_0\dTh_0\label{s2_ineq}.
\end{align}
Because $|\th_i - |v_i|^p | \leq ce^{-\d t}$ and $|\barth_i -|\barv_i|^p| \leq ce^{-\d t}$, the first integral on the right hand side of \eqref{s2_ineq} is less than or equal to $ce^{-\d t} \cK$. Then, we apply the H\"older's inequality and the boundedness of $|\barv_i|$ and $|v_i|$ to the second integral to obtain
\begin{align}\label{s2}
    \cS_2 \leq c( e^{-\d t}\cK + \cK^{1/2}\cH^{1/2}).
\end{align}
 Combining \eqref{s1} and \eqref{s2} we get
\begin{equation}\label{e:K}
\ddt \cK \leq c e^{-\d t}\cK^{1/2}  ( \cK^{1/2} + 1 + \cP^{1/2}) + \cK^{1/2}\cH^{1/2}.
\end{equation}

Let us now turn to the characteristic parameters term  $\cC$:
\[
\begin{split}
\ddt \cC & =  \int_{\O^{N}}\sum_{i=1}^N (\th_i - \barth_i) \cdot \left( \frac1N \sum_{k=1}^N \phi(x_i - x_k)(\th_k - \th_i) - \int_{\O}\phi(\barx_i-y) (\eta-\barth_i) f(y,w,\eta,t) \dy \dw\de \right) \\
& \hspace{3in} \times f_0^{\otimes N} \dX_0 \dV_0\dTh_0 \\
&: = I_1 + I_2+ I_3,
\end{split}
\]
where 
\[
\begin{split}
    I_1& = \int_{\O^{N}}\sum_{i=1}^N (\th_i - \barth_i) \cdot \frac1N \sum_{k=1}^N [\phi(x_i - x_k)- \phi(\barx_i - \barx_k)](\th_k - \th_i) \, f_0^{\otimes N} \dX_0 \dV_0\dTh_0,\\
I_2& = \int_{\O^{N}}\sum_{i=1}^N (\th_i - \barth_i) \cdot \frac1N \sum_{k=1}^N \phi(\barx_i - \barx_k)[(\th_k-\barth_k) - (\th_i - \barth_i)] \, f_0^{\otimes N} \dX_0 \dV_0\dTh_0,\\
I_3 & = \int_{\O^{N}}\sum_{i=1}^N (\th_i - \barth_i) \cdot \left(\frac1N \sum_{k=1}^N \phi(\barx_i - \barx_k)(\barth_k - \barth_i) - \int_{\O}\phi(\barx_i-y) (\eta-\barth_i) f(y,w,t) \dy \dw \de \right) \\
&\hspace{3in} \times f_0^{\otimes N} \dX_0 \dV_0 \dTh_0.
\end{split}
\]
We have $I_2 \leq 0$ because
\[
\begin{split}
    I_2 &= \int_{\O^{N}} \frac1N \sum_{i,k=1}^N \phi(\barx_i - \barx_k)[(\th_i - \barth_i) \cdot(\th_k-\barth_k) - |\th_i - \barth_i|^2] \, f_0^{\otimes N} \dX_0 \dV_0\dTh_0\\
    & = - \int_{\O^{N}} \dfrac{1}{2N} \sum_{i,k=1}^N \phi(\barx_i - \barx_k)|(\th_i-\barth_i) - (\th_k - \barth_k)|^2 \, f_0^{\otimes N} \dX_0 \dV_0\dTh_0.
\end{split}
\]
For $I_1$, we obtain, using \prop{p:flockR},
\[
\begin{split}
    |I_1|^2 &\leq 2\cC \int_{\O^{N}} \sum_{i=1}^N\left|\frac1N \sum_{k=1}^N [\phi(x_i - x_k)- \phi(\barx_i - \barx_k)](\th_k - \th_i) \right|^2\,f_0^{\otimes N} \dX_0 \dV_0 \dTh_0\\
    &\leq 2|\n\phi|_{\infty}^2\cC \int_{\O^{N}} \sum_{i=1}^N\left( \frac1N \sum_{k=1}^N |(x_i - x_k)  - (\barx_i - \barx_k)||\th_k - \th_i|\right)^2\,f_0^{\otimes N} \dX_0 \dV_0 \dTh_0\\
    & \leq ce^{-2\d t} \cC \int_{\O^{N}} \sum_{i=1}^N\left( \frac1N \sum_{k=1}^N (|x_i - \barx_i|  + |x_k - \barx_k|)\right)^2\,f_0^{\otimes N} \dX_0 \dV_0 \dTh_0\\
    &\leq ce^{-2\d t} \cC \int_{\O^{N}} \sum_{i=1}^N\left( |x_i - \barx_i|^2  + \frac1N \sum_{k=1}^N|x_k - \barx_k|^2\right)\,f_0^{\otimes N}\dX_0 \dV_0 \dTh_0 \\
    &\quad = ce^{-2\d t} \cC \cP.
\end{split}
\]
Thus, 
\begin{equation} 
    |I_1| \leq c e^{-\d t} \cC^{1/2}\cP^{1/2}.
\end{equation}
For $I_3$, we have 
\[
\begin{split}
|I_3|^2 &\leq 2\cC \int_{\O^{N}}\sum_{i=1}^N  \left|\frac1N \sum_{k=1}^N \phi(\barx_i - \barx_k)(\barth_k - \barth_i) - \int_{\O}\phi(\barx_i-y) (\eta-\barth_i) f(y,w,\eta,t) \dy \dw\de \right|^2 \\
&\hspace{3in} \times f_0^{\otimes N}(X_0,V_0,\Th_0) \dX_0 \dV_0 \dTh_0\\
&\quad =  2\cC \int_{\O^{N}}\sum_{i=1}^N  \left|\frac1N \sum_{k=1}^N \phi(\barx_i - \barx_k)(\barth_k - \barth_i) - \int_{\O}\phi(\barx_i-y) (\eta-\barth_i) f(y,w,\eta,t) \dy \dw\de \right|^2 \\
&\hspace{3in} \times f^{\otimes N}(\barX,\barV,\barTh,t) \dbarX\dbarV\dbarTh  \\
&\quad = 2\cC N\int_{\O^{N}}  \left|\frac1N \sum_{k=1}^N \phi(\barx_1 - \barx_k)(\barth_k - \barth_1) - \int_{\O}\phi(\barx_1-y) (\eta-\barth_1) f(y,w,\eta,t) \dy \dw\de \right|^2 \\
&\hspace{3in} \times f^{\otimes N}(\barX,\barV,\barTh,t) \dbarX\dbarV\dbarTh\\
&\quad\leq  2\cC N \frac4N \sup\limits_{(\barx,\barv,\barth),(\barx',\barv',\barth') \in \supp f_t} |\phi(\barx-\barx')(\barth - \barth')|^2 \leq c \cC e^{-2\d t}.
\end{split}
\]
Here in the penultimate step we used again \cite[Lemma 3.3]{NP2021-orig}.
Therefore, 
\begin{equation}
 |I_3|\leq ce^{-\d t} \cC^{1/2}.   
\end{equation}

Combining  the three estimates for $I_1, I_2, I_3$, we obtain
\begin{equation}\label{e:H}
\ddt \cC \leq ce^{-\d t}(1+\cP^{1/2}) \cC^{1/2}.
\end{equation}
Setting $x = 1+ \cP^{1/2}$, $y = \cK^{1/2}, z = \cC^{1/2}$. By \eqref{e:PK}, \eqref{e:K} and \eqref{e:H} we obtain the system of  ODIs:
\begin{equation}\label{e:xyz_sys}
\left\{\begin{split}
\dot{x} &\leq y,\quad x_0 = 1\\
\dot{y}& \leq ce^{-\d t} (x+ y) + cz,\quad y_0 = 0\\
\dot{z} &\leq c e^{-\d t} x, \quad z_0 = 0.
\end{split}\right.
\end{equation}
\begin{claim}\label{xyz_est}
For any nonnegative solution $(x,y,z)$ to \eqref{e:xyz_sys}, there exists a constant $C$ depending on $c, \d$ such that 
\begin{equation}
    x\leq 1 +C t^2, \quad y \leq C t, \quad z \leq C \min \{1, t\}.
\end{equation}
\end{claim}
\begin{proof}[Proof of the Claim \ref{xyz_est}]
Fix $\e, \tau > 0$ to be chosen later. We have
\[
\left\{\begin{split}
&\ddt (\e x^2) \leq 2\e x y \leq \sqrt{\e}(\e x^2 + y^2)\\
&\ddt y^2 \leq ce^{-\d t}(2xy + 2y^2) + 2c y z \leq ce^{-\d t} \left[\dfrac{1}{\sqrt{\e}}(\e x^2 + y^2) + 2y^2\right] + \dfrac{c}{\sqrt{\t}}(y^2 + \t z^2)\\
&\ddt (\t z^2) \leq 2\t ce^{-\d t}xz \leq \dfrac{c \sqrt{\t}e^{-\d t}}{\sqrt{\e}} (\e x^2 +\t z^2)
\end{split}\right.
\]
It implies that
\begin{align*}
    \ddt(\e x^2 + y^2 + \t z^2) \leq c(\t, \e) e^{-\d t} ( \e x^2 + y^2 + \t z^2) + \left(\sqrt{\e} + \dfrac{c}{\sqrt{\t}}\right)(\e x^2 + y^2 + \t z^2).
\end{align*}
Applying the Gr\"onwall's lemma we get
\[
 \e x^2 + y^2 + \t z^2 \leq \e \exp\left( \left(\sqrt{\e} + \dfrac{c}{\sqrt{\t}}\right) t + \dfrac{c(\e,\t)}{\d}(1- e^{-\d t})\right)\leq \e \exp\left( \left(\sqrt{\e} + \dfrac{c}{\sqrt{\t}}\right) t + \dfrac{c(\e,\t)}{\d}\right).
\]
Now choosing $\e = \d ^2/4, \t = 4c^2/\d^2$, we obtain
\begin{equation*}\label{x}
   x\lesssim e^{\d t/2}. 
\end{equation*}
Plugging it into the third equation in \eqref{e:xyz_sys} and solving for $z$ we have
\begin{equation*}\label{z}
  z \leq c\int_0^t e^{-\d s/2}\ds \leq C\min \{ 1, t\}.  
\end{equation*}
Substituting $x,z$ into the second equation in  \eqref{e:xyz_sys} we have
\begin{align*}
    \ddt y \leq ce^{-\d t} y + ce^{-\d t/2} + C\min (1,t).
\end{align*}
It implies that 
\begin{equation*}
    y \leq C t.
\end{equation*}
Hence, by the first equation in \eqref{e:xyz_sys} we get
\begin{equation*}
    x \leq 1 + Ct^2.
\end{equation*}
\end{proof}
Claim \ref{xyz_est} follows that
\begin{equation*}
    \cP \leq Ct^4, \quad \cK \leq Ct^2, \quad \cC \leq C\min\{1, t^2\}.
\end{equation*}
On the other hand, in view of the global estimates on the support of the flock,  $\cP \leq C N$. Due to the alignment we also have $\cK \leq C N$. Therefore,
\begin{equation*}
\cP + \cK + \cC \leq C \min \{N, t^4\},
\end{equation*}
as desired.
\end{proof}

\begin{remark}\label{r:k=0}
Our final remark concerns the case $\k = 0$. This represents the system with "frozen" characteristic parameters $\th$. In opinion dynamics such system can be interpreted as a non-cooperative game where "players" come with their fixed convictions $\th$'s but may change their opinions $v$'s to achieve a consensus. In the discrete case this situation was examined in detail in \cite{LRS-friction} where the consensus was identified as a Nash equilibrium. The equilibrium is unique, stable, and is also a global  attractor for the system. While the kinetic version of such result would be highly desirable to achieve -- this could be interpreted as a dynamics of infinitely many players -- we leave this question to a future research.  At this point we note that in the case $\k=0$ no alignment dynamics is possible, however the maximum principle obtained in \lem{v:bdd}  and \lem{l:wmp} still holds. The Grassmannian reduction still works also to show that  the support of any sectorial solution $f_t$ narrows down to a kinetic ray $ \R_+ v_\infty$ for some $v_\infty\in \S^{n-1}$. It is therefore an essentially unidirectional flow. 

Since no global flocking information is available in this case applying our analysis gives the same exponential rate as in Natalini and Paul's estimate \eqref{e:W2det}.

\end{remark}


\begin{thebibliography}{10}

\bibitem{BCC2011}
Fran\c{c}ois Bolley, Jos\'{e}~A. Ca\~{n}izo, and Jos\'{e}~A. Carrillo.
\newblock Stochastic mean-field limit: non-{L}ipschitz forces and swarming.
\newblock {\em Math. Models Methods Appl. Sci.}, 21(11):2179--2210, 2011.

\bibitem{CFRT2010}
J.~A. Carrillo, M.~Fornasier, J.~Rosado, and G.~Toscani.
\newblock Asymptotic flocking dynamics for the kinetic {C}ucker-{S}male model.
\newblock {\em SIAM J. Math. Anal.}, 42(1):218--236, 2010.

\bibitem{CS2007a}
F.~Cucker and S.~Smale.
\newblock Emergent behavior in flocks.
\newblock {\em IEEE Trans. Automat. Control}, 52(5):852--862, 2007.

\bibitem{CS2007b}
F.~Cucker and S.~Smale.
\newblock On the mathematics of emergence.
\newblock {\em Jpn. J. Math.}, 2(1):197--227, 2007.

\bibitem{Ha-friction}
S.-Y. Ha, T.~Ha, and J.-H. Kim.
\newblock Asymptotic dynamics for the {C}ucker-{S}male-type model with the
  {R}ayleigh friction.
\newblock {\em J. Phys. A}, 43(31):315201, 19, 2010.

\bibitem{Ha-stability}
S.-Y. Ha, J.~Kim, and X.~Zhang.
\newblock Uniform stability of the {C}ucker-{S}male model and its application
  to the mean-field limit.
\newblock {\em Kinet. Relat. Models}, 11(5):1157--1181, 2018.

\bibitem{HL2009}
S.-Y. Ha and J.-G. Liu.
\newblock A simple proof of the {C}ucker-{S}male flocking dynamics and
  mean-field limit.
\newblock {\em Commun. Math. Sci.}, 7(2):297--325, 2009.

\bibitem{HT2008}
S.-Y. Ha and E.~Tadmor.
\newblock From particle to kinetic and hydrodynamic descriptions of flocking.
\newblock {\em Kinet. Relat. Models}, 1(3):415--435, 2008.

\bibitem{LRS-friction}
Daniel Lear, David~N. Reynolds, and Roman Shvydkoy.
\newblock Grassmannian reduction of cucker-smale systems and dynamical opinion
  games.
\newblock {\em Discrete Contin. Dyn. Syst.}, 41(12):5765--, 2021.

\bibitem{NPchemo}
Roberto Natalini and Thierry Paul.
\newblock The mean-field limit for hybrid models of collective motions with
  chemotaxis.
\newblock 2021.

\bibitem{NP2021-orig}
Thierry~Paul Roberto~Natalini.
\newblock On the mean field limit for cucker-smale models.
\newblock {\em Discrete Contin. Dyn. Syst. B}, 2021.

\bibitem{Sbook}
Roman Shvydkoy.
\newblock {\em Dynamics and analysis of alignment models of collective
  behavior}.
\newblock Ne\v{c}as Center Series. Birkh\"{a}user/Springer, Cham, [2021]
  \copyright 2021.

\bibitem{Sznitman}
Alain-Sol Sznitman.
\newblock Topics in propagation of chaos.
\newblock In {\em \'{E}cole d'\'{E}t\'{e} de {P}robabilit\'{e}s de
  {S}aint-{F}lour {XIX}---1989}, volume 1464 of {\em Lecture Notes in Math.},
  pages 165--251. Springer, Berlin, 1991.

\end{thebibliography}

\end{document}